\documentclass[12pt]{article}

\newtheorem{theorem}{Theorem}
\newtheorem{corollary}{Corollary}
\newtheorem{lemma}{Lemma}

\newtheorem{definition}{Definition}

\usepackage{epsfig}
\usepackage{url}
\usepackage{amssymb}
\usepackage{amsmath}
\usepackage{amsfonts}

\def\qed{\begin{flushright} $\Box$ \end{flushright}}

\def\Dbar{\leavevmode\lower.6ex\hbox to 0pt{\hskip-.23ex \accent"16\hss}D}

\def\bZ{{\mbox{\bf Z}}}
\def\bC{{\mbox{\bf C}}}

\def\bT{{\mbox{\bf T}}}
\def\hG{{\hat{G}}}
\def\paf{{\mbox{\rm PAF}}}
\def\psd{{\mbox{\rm PSD}}}
\def\dft{{\mbox{\rm DFT}}}
\def\Mat{{\mbox{\rm Mat}}}
\newcommand{\nc}{\newcommand}
\nc{\cP}{{\cal P}}
\nc{\cR}{{\cal R}}

\begin{document}

{\bf\LARGE
\begin{center}
Algorithms for difference families in finite abelian groups
\end{center}
}

{\Large
\begin{center}
Dragomir {\v{Z}}. {\Dbar}okovi{\'c}\footnote{University of Waterloo,
Department of Pure Mathematics and Institute for Quantum 
Computing, Waterloo, Ontario, N2L 3G1, Canada
e-mail: \url{djokovic@math.uwaterloo.ca}}, Ilias S.
Kotsireas\footnote{Wilfrid Laurier University, Department of Physics
\& Computer Science, Waterloo, Ontario, N2L 3C5, Canada, e-mail:
\url{ikotsire@wlu.ca}}
\end{center}
}

\begin{abstract}
Our main objective is to show that the computational methods that we previously developed to search for difference families in cyclic groups can be fully  extended to the more general case of arbitrary finite abelian groups. In particular the power density $(\psd)$-test and the method of compression can be used to help the search.
\end{abstract}

{\em Keywords:}
Difference families, discrete Fourier transform, periodic autocorrelation, power spectral density, Goethals--Seidel array.

2010 Mathematics Subject Classification: 05B10, 05B20.

\section{Difference families}

Let $G$ be a finite abelian group of order $v$. We write it 
multiplicatively and denote its identity element by $e$. 

\begin{definition} \label{def:diff-fam}
We say that an ordered $t$-tuple $(X_1,X_2,\ldots,X_t)$, 
$t\ge1$, of proper nonempty subsets of $G$ is a 
{\em difference family} if the sets
$S_a=\{(x,ax,i):x,ax\in X_i,i=1,2,\ldots,t\}$ with  
$a\in G\setminus\{e\}$ have the same cardinality. 
In that case we denote this cardinality by $\lambda$ and we refer to
\begin{equation} \label{eq:param}
(v;k_1,k_2,\ldots,k_t;\lambda),
\end{equation}
where $k_i=|X_i|$, as the {\em parameter set} of this 
difference family.
\end{definition}

A simple counting argument shows that these parameters must 
satisfy the equation

\begin{equation} \label{eq:lambda}
\sum_{i=1}^t k_i(k_i-1)=\lambda(v-1).
\end{equation}

We shall also need an additional parameter, $n$, defined 
by the equation
\begin{equation} \label{eq:par-n}
n=\sum_{i=1}^t k_i-\lambda.
\end{equation}

Let us assume that $(X_1,X_2,\ldots,X_t)$ is a difference 
family in $G$ with the parameters displayed above. Then 
we also say that the sets $X_i$ are its {\em base blocks}.
There are two simple types of transformations which we can perform on the difference families in $G$ with the fixed 
index set $\{1,2,\ldots,t\}$ and the fixed value of the 
parameter $n$.

For the first type let $\pi$ be a permutation of the index set 
$\{1,2,\ldots,t\}$. Then the $t$-tuple 
$(X_{\pi 1},X_{\pi 2},\ldots,X_{\pi t})$ 
is also a difference family. Its parameter set is obtained 
from (\ref{eq:param}) by substituting $\pi k_i$ for $k_i$ 
for each index $i$. The parameter $\lambda$ does not change.

For the second type we select an index, say $j$, and replace 
$X_j$ with its complement $G\setminus X_j$. We obtain again 
a difference family. Its parameter set is obtained from 
(\ref{eq:param}) by substituting $v-k_j$ for $k_j$ and 
$\lambda+v-2k_j$ for $\lambda$.

By performing a finitely many transformations of the two 
types described above, we can replace the original difference family with one whose parameter set satisfies the following
additional conditions
\begin{equation} \label{uslovi}
v/2 \ge k_1\ge k_2\ge \cdots \ge k_t \ge 1.
\end{equation}
Note that this implies that $v \ge 2$.

\section{Difference families and the group algebra}
\label{sec:groupalg}

Let $\cR$ be the group algebra of $G$ over the complex numbers, 
$\bC$. The elements of $\cR$ are formal linear combinations $\sum_{x\in G} c_x x$, with complex coefficients $c_x\in\bC$. Thus $G$ is a vector space basis of $\cR$, and $\cR$ has dimension $v$. The linear map $\varepsilon:\cR\to\bC$ 
such that $\varepsilon(x)=1$ for all $x\in G$ is an algebra homomorphism known as the {\em augmentation}. Note that
$$
\varepsilon(\sum_{x\in G} c_x x)=\sum_{x\in G} c_x.
$$

The algebra $\cR$ has an involution, ``$*$'', which acts on the scalars as the complex conjugation and acts 
as the inversion map on group elements. Thus
$$
\left( \sum_{x\in G} c_x x\right)^* =
\sum_{x\in G} \bar{c}_x x^{-1}, \quad c_x\in\bC.
$$

For any subset $X\subseteq G$, by abuse of notation, we also 
denote by $X$ the sum of all elements of $X$ in $\cR$. It will be clear from the context which of these two meanings is used. 
For instance, we have 
$$
X=\sum_{x\in X} x.
$$
The $X$ on the lefthand side is an element of $\cR$, while on the 
righthand side it is a subset of $G$.

In section \ref{sec:Classes} we shall use the symmetric and 
skew subsets of $G$. We define them as follows. 
A subset $X\subseteq G$ is {\em symmetric} if $X^*=X$. 
A subset $Y\subseteq G$ is {\em skew} if $G$ is a disjoint 
union of $Y$, $Y^*$ and $\{e\}$.

For any element $X\in \cR$ we define its {\em norm}, $N(X)$, 
by $N(X)=XX^*$. 
The proof of the following lemma is straightforward and we omit it.
\begin{lemma} \label{le:folk}
Let $(X_1,X_2,\ldots,X_t)$ be a $t$-tuple of proper nonempty 
subsets of $G$ with cardinalities $k_i=|X_i|$, and let $\lambda$ be a nonnegative integer. Then $(X_1,X_2,\ldots,X_t)$ is a difference family in $G$ with the parameter set 
$(v;k_1,k_2,\ldots,k_t;\lambda)$ if and only if 
\begin{equation} \label{eq:zbirNi}
\sum_{i=1}^t N(X_i)=n\cdot e+\lambda G,
\end{equation}
where $n$ is defined by (\ref{eq:par-n}).
\end{lemma}

To any function $f:G\to\bC$ we assign an element, $a_f\in \cR$, 
by seting
$$a_f=\sum_{x\in G} f(x)x.$$
The conjugate $\bar{f}$ of $f$ is defined by 
$\bar{f}(x)=\overline{f(x)}$, $x\in G$.
 
The {\em periodic autocorrelation function} $\paf_f$ of $f$ is 
defined by
\begin{equation} \label{def:paf}
\paf_f(x)=\sum_{y\in G} f(xy)\bar{f}(y),\quad x\in G.
\end{equation}
We claim that
\begin{equation} \label{eq:paf-norm}
N(a_f)=\sum_{x\in G} \paf_f(x)x.
\end{equation}
Indeed, this follows from the identities
\begin{eqnarray*}
N(a_f) &=& N\left( \sum_{x\in G} f(x)x \right) \\
&=& \sum_{x,y\in G} f(x)\bar{f}(y)xy^{-1} \\
&=& \sum_{z\in G}\left( \sum_{y\in G}f(yz)\bar{f}(y)\right)z \\
&=& \sum_{z\in G} \paf_f(z)z.
\end{eqnarray*}

\section{Characters of finite abelian groups}

In this section we recall some well known facts. For more
details see for instance \cite{AT,BL}.

Let $L^2(G)$ denote the finite-dimensional complex Hilbert 
space of all complex-valued functions on $G$ with the 
inner product
\begin{equation} \label{eq:inn-prod}
\langle f,g \rangle=\sum_{x\in G} f(x)\bar{g}(x).
\end{equation}
For $a\in G$ let $\delta_a:G\to\bC$ be the function 
defined by $\delta_a(x)=1$ if $x=a$ and $\delta_a(x)=0$ 
otherwise. These delta functions form an orthonormal 
(o.n.) basis of $L^2(G)$.

A {\em character} $\chi$ of $G$ is a group homomorphism 
$G\to \bT$, where $\bT=\{z\in\bC:|z|=1\}$ is the circle group, i.e., the group of complex numbers of modulus 1 (with multiplication as the group operation). If $\chi$ and $\psi$ are two characters of $G$, then their product $\chi\psi$ defined, as usual, by $\chi\psi(x)=\chi(x)\psi(x)$, $x\in G$, is also a character of $G$. It follows easily that the characters of $G$ form an abelian group, called the {\em dual group} of $G$, which we denote by $\hG$. Its identity element is the {\em trivial character}, $\theta$, which maps all elements $x\in G$ to $1\in\bT$. The group $\hG$ is finite, and it has the same order as 
$G$. Moreover, these two groups are isomorphic, $\hG\cong G$.

If $\alpha:H\to G$ is a homomorpism of finite abelian groups
and $\chi$ is a character of $G$ then the composite 
$\chi\circ\alpha:H\to\bT$ is a character of $H$. 
In particular, the map $\alpha:G\to G$ defined by 
$\alpha(x)=x^{-1}$, $x\in G$, is an automorphism of $G$
and we have $\chi\circ\alpha(x)=\chi(x^{-1})=\chi(x)^{-1}=\bar{\chi}(x)$. Thus if $\chi$ is a character of $G$ then 
its complex conjugate $\bar{\chi}$ is also a character of $G$.

We claim that if $\chi$ is a nontrivial character of $G$ then
\begin{equation} \label{eq:sum-x}
\sum_{x\in G} \chi(x)=0, \quad \chi\ne\theta.
\end{equation}
Indeed, there exists $a\in G$ such that $\chi(a)\ne1$. Then 
$\chi(a)\sum_{x\in G} \chi(x)=\sum_{x\in G} \chi(ax)=\sum_{x\in G} \chi(x)$ and our claim follows. 

Next we claim that if $\chi$ and $\psi$ are different characters 
of $G$ then they are orthogonal, i.e., 
\begin{equation} \label{eq:ortog}
\langle \chi,\psi \rangle=0,\quad \chi\ne\psi. 
\end{equation}
Indeed, we have 
$\langle \chi,\psi \rangle=\sum_x \chi(x)\bar{\psi}(x)
=\sum_x \chi\bar{\psi}(x)=0$ since the character 
$\chi\bar{\psi}$ is nontrivial.

For $\chi\in\hG$ we have $\langle \chi,\chi \rangle=
\sum_{x\in G} \chi(x)\bar{\chi}(x)=\sum_{x\in G} 1=|G|=v$.
Thus the functions $\chi/\sqrt{v}$, with $\chi\in\hG$, form 
another o.n. basis of $L^2(G)$.

\section{Discrete Fourier transform, \dft} \label{sec:dft}

The {\em discrete Fourier transform} of a function $f\in L^2(G)$ is the function $\dft_f$, also written as $\dft(f)$, in 
$L^2(\hG)$ defined by the formula

\begin{equation} \label{def:dft}
\dft_f(\chi)=\sum_{x\in G} f(x)\bar{\chi}(x)
=\langle f,\chi \rangle, \quad \chi\in\hG.
\end{equation}

Thus we have a linear map $\dft:L^2(G)\to L^2(\hG)$ to which we 
refer as the {\em discrete Fourier transform} on $G$. For the proofs of the following four basic properties of this transform we again refer to \cite{AT}.

(a) $\dft$ is a vector space isomorphism.

(b) $\dft_{f*g}=\dft_f\cdot\dft_g$ where $f*g$ is the {\em convolution} of $f$ and $g$ defined by the formula
$$ f*g\;(x)=\sum_{y\in G} f(y)g(xy^{-1}). $$

(c) For $f\in L^2(G)$, 
$$f=\frac{1}{v} \sum_{\chi\in\hG}\dft_f(\chi)\chi.$$

(d) Define the inner product in $L^2(\hG)$ by 
$\langle H,K \rangle= \sum_{\chi\in\hG} H(\chi)\bar{K}(\chi)$. 
Then
$$\|\dft_f\|^2=v\|f\|^2,\quad f\in L^2(G).$$


The property (c) is known as the {\em inversion formula} as 
it shows how to recover the function $f$ from its Fourier 
transform $\dft_f$. It follows from (c) that the {\em inverse 
Fourier transform}, $\dft^{-1}$, is given by the formula

\begin{equation} \label{def:dft-inv}
\dft^{-1}(\phi)= \frac{1}{v} \sum_{\chi\in\hG}\phi(\chi)\chi,
\quad \phi\in L^2(\hG).
\end{equation}

The property (d) shows that the linear map 
$(1/\sqrt{v})\dft: L^2(G)\to L^2(\hG)$ is an isometry.

As an example let us compute the $\dft$ of the cyclic group 
$G=\langle g \rangle$ of order $v$ with a generator $g$. 
For each $j\in\{0,1,\ldots,v-1\}$ there is a unique character 
$\chi_j$ of $G$ such that $\chi_j(g)=\omega^j$, where 
$\omega=e^{2\pi i/v}$. Hence we have 
$\hat{G}=\{\chi_j:j=0,1,\ldots,v-1\}$.
Note that the trivial character is $\theta=\chi_0$.
Let $f\in L^2(G)$ be arbitrary. Then
\begin{eqnarray*}
\dft_f(\chi_j) &=& \langle f,\chi_j \rangle \\
&=& \sum_{k=0}^{v-1} f(g^k)\bar{\chi}_j(g^k) \\
&=& \sum_{k=0}^{v-1} f(g^k)\overline{\omega^{jk}} \\
&=& \sum_{k=0}^{v-1} \omega^{-jk} f(g^j).
\end{eqnarray*}

The {\em power spectral density} of a function $f\in L^2(G)$ 
is the function $\psd_f\in L^2(\hG)$ defined by
\begin{equation} \label{def:psd}
\psd_f(\chi)=\left| \dft_f(\chi) \right|^2.
\end{equation}

In various places, in the case when $G$ is cyclic, the following simple lemma is referred to as the ``Wiener-Khinchin theorem'', see \cite[Theorem 1]{FGS:2001} and its references, and also 
\cite{DK:DCC:2014}.

\begin{lemma} \label{le:W-H}
$\psd_f = \dft(\paf_f), \quad f\in L^2(G).$
\end{lemma}
\begin{proof} For $\chi\in\hG$ we have
\begin{eqnarray*}
\psd_f(\chi) &=& 
\left| \sum_{x\in G} f(x)\bar{\chi}(x) \right|^2 \\
&=& 
\sum_{x,y\in G} f(x)\bar{f}(y)\bar{\chi}(x)\bar{\chi}(y^{-1}) \\
&=& 
\sum_{x,y\in G} f(x)\bar{f}(y)\bar{\chi}(xy^{-1}) \\
&=& 
\sum_{z\in G} \left( \sum_{y\in G} f(yz)\bar{f}(y) \right) \bar{\chi}(z) \\
&=& \sum_{z\in G} \paf_f(z)\bar{\chi}(z) \\
&=& \dft(\paf_f)(\chi).
\end{eqnarray*}
\end{proof}
\qed

\section{Complementary functions} \label{sec:compl}

The notion of complementary sequences plays an important role
in combinatorics, see e.g. \cite{KK:2007,SY:1992}.
By analogy, we define the complementary functions on $G$.

\begin{definition} \label{def:compl}
We say that the functions $f_1,f_2,\ldots,f_t\in L^2(G)$ 
are {\em complementary} if
\begin{equation} \label{eq:compl=}
\sum_{i=1}^t \paf_{f_i}=(\alpha_0-\alpha)\delta_e+\alpha\theta.
\end{equation}
or, equivalently, 
\begin{equation} \label{eq:compl}
\sum_{i=1}^t \paf_{f_i}(s)=
\begin{cases}
\alpha_0,& {\rm if}~ s=e;\\
\alpha,& {\rm otherwise},
\end{cases}
\end{equation}
for some constants $\alpha_0$ and $\alpha$ 
(the $\paf$-constants).
\end{definition}

Note that
$$
\alpha_0=\sum_{i=1}^t \sum_{x\in G} |f_i(x)|^2
=\sum_{i=1}^t \|f_i\|^2.
$$
In the special case when the $f_i$ take values in 
$\{\pm1\}$, we have $\alpha_0=tv$.

By analogy with $G$, we introduce the notation $\hat{\theta}$ 
for the trivial character of the group $\hG$, and denote by 
$\delta_\theta$ the function on $\hG$ which takes value 1 
at $\theta$ and value 0 at all other characters $\chi$ of $G$.
It is easy to verify that
\begin{equation} \label{eq:delta-theta}
\dft_{\delta_e}=\hat{\theta}, \quad  
\dft_\theta=v\delta_\theta.
\end{equation}

\begin{theorem} \label{thm:paf=psd}
The functions $f_1,f_2,\ldots,f_t\in L^2(G)$ are 
complementary with \paf-constants $\alpha_0$ and $\alpha$ 
if and only if
\begin{equation} \label{eq:compl-psd=}
\sum_{i=1}^t \psd_{f_i}
=(\beta_0-\beta)\delta_\theta+\beta\hat{\theta}.
\end{equation}
or, equivalently, 
\begin{equation} \label{eq:compl-psd}
\sum_{i=1}^t \psd_{f_i}(\chi)=
\begin{cases}
\beta_0,& {\rm if}~ \chi=\theta;\\
\beta,& {\rm otherwise},
\end{cases}
\end{equation}
where the constants $\beta_0$ and $\beta$ (the $\psd$-constants) are given by
\begin{equation} \label{eq:konstante}
\beta_0=\alpha_0+(v-1)\alpha, \quad \beta=\alpha_0-\alpha. 
\end{equation}
\end{theorem}

\begin{proof} 
Since at the points $\theta$ and $\chi\ne\theta$ the function
$\delta_\theta$ takes values 1 and 0, respectively, while 
the function $\hat{\theta}$ takes value 1 at all points 
$\chi\in\hG$, we deduce that the equations (\ref{eq:compl-psd=}) 
and (\ref{eq:compl-psd}) are equivalent.

Assume that the functions $f_1,f_2,\ldots,f_t$ are complementary, i.e., that (\ref{eq:compl=}) holds. 
By applying \dft~ to (\ref{eq:compl=}) 
and by using Lemma \ref{le:W-H} and the equations 
(\ref{eq:delta-theta}), we obtain that
$$
\sum_{i=1}^t \psd_{f_i}
=(\alpha_0-\alpha)\hat{\theta}+\alpha v\delta_\theta.
$$
Hence, the equation (\ref{eq:compl-psd=}) holds where the constants $\beta_0$ and $\beta$ are defined as in  (\ref{eq:konstante}).

To prove the converse, we just have to apply the inverse 
Fourier transform to the equation (\ref{eq:compl-psd=}).
\end{proof}
\qed

For any subset $X\subseteq G$ we define a $\{\pm1\}$-valued 
function $f_X$ on $G$ as follows: $f_X(x)$ is equal to $-1$ 
if $x\in X$ and is equal to $+1$ otherwise. We also say that 
the function $f_X$ is {\em associated with} $X$.

In the next theorem we show that each difference family with $t$ base blocks gives $t$ complementary functions having values in $\{\pm 1\}$, and we compute their $\paf$-constants. Their 
$\psd$-constants can be computed by using the formulas 
(\ref{thm:paf=psd}).

\begin{theorem} \label{thm:compl-fun}
Let $(X_1,X_2,\ldots,X_t)$ be a difference family in $G$ with 
parameter set (\ref{eq:param}) and let $f_i=f_{X_i}$,  
$i=1,2,\ldots,t$, be their associated functions. Then 

(a) \begin{equation} \label{eq:glavna}
\sum_i N(a_{f_i})=4n\cdot e+(tv-4n)G;
\end{equation}

(b) the $f_i:G\to \{\pm1\}$ are non-constant complementary functions with $\paf$-constants 
\begin{equation} \label{paf-konst}
\alpha_0=tv, \quad \alpha=tv-4n;
\end{equation}

(c) \begin{equation} \label{eq:sum-kvad}
\sum_{i=1}^t (v-2k_i)^2=4n+v(tv-4n).
\end{equation}
\end{theorem}
\begin{proof}
If we view the $X_i$ and $G$ as elements of $\cR$, then we have $xG=G$ for $x\in G$, $X_iG=GX_i=k_iG$ for each $i$, and $GG=vG$.
Since $a_{f_i}=G-2X_i$, we have
\begin{equation} \label{eq:Nafi}
N(a_{f_i})=(G-2X_i)(G-2X_i^*)=(v-4k_i)G+4N(X_i).
\end{equation}

The assertion (a) follows by adding up these equations and by 
applying Lemma \ref{le:folk}.

The assertion (b) follows from (a) by using the formula 
(\ref{eq:paf-norm}) and by comparing the coefficients of 
$x\in G$ on both sides.

The assertion (c) follows from (a) by applying the augmentation map $\varepsilon$. 
\end{proof}
\qed

The converse of the assertion (b) is also valid.
\begin{theorem} \label{thm:converse}
Let $f_i:G\to \{\pm1\}$, $i=1,2,\ldots,t$, be non-constant complementary functions with $\paf$-constants 
$\alpha_0$ and $\alpha$. 
Then the $X_i=\{x\in G:f_i(x)=-1\}$, $i=1,2,\ldots,t$, 
form a difference family of $G$ with parameters
$(v;k_1,k_2,\ldots,k_t;\lambda)$ where $k_i=|X_i|$ and 
$\lambda=\sum k_i-(tv-\alpha)/4$.
\end{theorem}
\begin{proof}
Since the functions $f_i$ are non-constant, the $X_i$ are 
proper non-empty subsets of $G$. 
By the hypothesis, the formula (\ref{eq:compl=}) holds true. 
Hence, for $x\in G$ we have
$$
\sum \paf_{f_i}(x)=(\alpha_0-\alpha)\delta_e(x)+\alpha.
$$
By multiplying this equation by $x$, summing up over all 
$x\in G$, and by using (\ref{eq:paf-norm}), we obtain that 
$$
\sum N(a_{f_i})=(\alpha_0-\alpha)e+\alpha G.
$$
Since $\alpha_0=tv$, by using (\ref{eq:Nafi}) we obtain that 
$$
\sum N(X_i)=\frac{tv-\alpha}{4}e
+\left(\sum k_i-\frac{tv-\alpha}{4} \right) G.
$$
It follows that $tv-\alpha$ is divisible by 4 and the 
assertion of the theorem follows from Lemma \ref{le:folk}. 
\end{proof}
\qed

\section{PSD-test} \label{sec:test}

Suppose that we want to search for a difference family 
$(X_1,X_2,\ldots,X_t)$ in $G$ having the parameter set
(\ref{eq:param}). An exhaustive search can be carried out 
only when the order, $v$, of $G$ is relatively small.
For larger $v$ one uses some randomized or heuristic 
procedure to generate candidates for the base blocks $X_i$. 
One can often improve such a procedure by using a test, 
known as the {\em PSD-test}, to discard some of the candidates 
for the set $X_i$. This test is based on Theorems 
\ref{thm:paf=psd} and \ref{thm:compl-fun}. 
First, the functions $f_i$ associated 
with $X_i$ are complementary with \paf-constants 
$\alpha_0=tv$ and $\alpha=tv-4n$. Second, for any nontrivial 
character $\chi$ of $G$ we must have 
$$
\sum_i \psd_{f_i}(\chi)=\beta=4n,\quad \chi\ne\theta. 
$$
(Recall that $n=\sum k_i -\lambda$.)

Since $\psd_{f_i}(\chi)=|\dft_{f_i}(\chi)|^2 \ge0$,
we can discard the candidate $X_i$ if for some 
$\chi\ne\theta$ we have $\psd_{f_i}(\chi)>4n$. 
In that case we say that $X_i$ (or $f_i$) fails the \psd-test.
This test is most effective when $t=2$.

There is another method for performing the \psd-test. 
For that we need to assign to $X\subseteq G$ the function 
$\Phi_X:G\to\bZ$ defined by
$$
N(X)=XX^*=\sum_{x\in G} \Phi_X(x)x.
$$
As $N(X)^*=N(X)$, we have $\Phi_X(x^{-1})=\Phi_X(x)$ for all 
$x\in G$.

Next we set $\Phi_i=\Phi_{X_i}$ for $i=1,2,\ldots,t$.
By (\ref{eq:paf-norm}) and (\ref{eq:Nafi}) we have
$$
\paf_{f_i}=(v-4k_i)\theta+4\Phi_i.
$$
Since $\dft_\theta=v\delta_\theta$, by applying the \dft-transform to the above equation and by using Lemma \ref{le:W-H}, 
we obtain that
$$
\psd_{f_i}=(v-4k_i)v\delta_\theta+4\dft_{\Phi_i}.
$$
By evaluating both sides at a nontrivial character $\chi$, 
we obtain that $\psd_{f_i}(\chi)=4\dft_{\Phi_i}(\chi)$. 
Since $\alpha_0=tv$, by adding up these equations and by using 
Theorem \ref{thm:paf=psd} we obtain that
$$
\sum_i \dft_{\Phi_i}(\chi)=n, \quad \chi\ne\theta.
$$ 
Hence, $X_i$ passes the \psd-test if and only if for all nontrivial characters $\chi$ of $G$ we have
$$
\dft_{\Phi_i}(\chi)\le n,\quad \chi\ne\theta.
$$
Since $\Phi_i(e)=k_i$, $\chi(e)=1$ and 
$\dft_{\Phi_i}(\chi)=\langle \Phi_i,\chi \rangle$ this inequality 
can be written as follows
$$
\sum_{x\in G\setminus \{e\} } 
\Phi_i(x) \Re \chi(x) \le n-k_i,\quad \chi\ne\theta.
$$
($\Re\chi$ denotes the real part of $\chi$.)

\section{Compression of complementary functions} 
\label{sec:compr}

Let $M$ be a subgroup of order $m$ of $G$ and $H=G/M$, 
the corresponding quotient group of order $d=v/m$. 
We denote by $\sigma$ the canonical map $G\to H$. 
Further, we denote by $\theta_H$ and $\hat{\theta}_H$ the 
trivial characters of $H$ and its dual group $\hat{H}$, 
respectively. Finally $\hat{\sigma}:\hat{H}\to\hat{G}$ 
will be the dual map of $\sigma$ defined by 
$\hat{\sigma}(\phi)=\phi\sigma$, $\phi\in\hat{H}$.

For any $f\in L^2(G)$ we define $f^M\in L^2(H)$ by
$$
f^M(xM)=\sum_{z\in M} f(xz), \quad x\in G.
$$
We say that $f^M$ is the {\em M-compression} of $f$ and that 
$m$ is the {\em compression factor}.
Note that $\bar{f}^M=\overline{f^M}$.
We choose a set $Y$ of coset representatives of $M$ in $G$. Then each $x\in G$ can be written uniquely as $x=yz$ with $y\in Y$ 
and $z\in M$.

\begin{lemma} \label{le:f^M}
For $f\in L^2(G)$ we have $\paf_{f^M}=(\paf_f)^M$, i.e.,
\begin{equation} \label{eq:f^M}
\paf_{f^M}(xM)=\sum_{z\in M} \paf_f(xz),\quad x\in G.
\end{equation}
\end{lemma}
\begin{proof}
We have
\begin{eqnarray*}
\paf_{f^M}(xM) &=& \sum_{y\in Y} f^M(xyM)\bar{f}^M(yM) \\
&=& \sum_{y\in Y} \sum_{p,w\in M} f(xyp)\bar{f}(yw).
\end{eqnarray*}
After setting $p=wz$, we obtain that
\begin{eqnarray*}
\paf_{f^M}(xM) &=& \sum_{z\in M} \sum_{y\in Y} \sum_{w\in M} f(xzyw)\bar{f}(yw) \\
&=& \sum_{z\in M} \paf_f(xz) \\
&=& (\paf_f)^M(xM).
\end{eqnarray*}
\end{proof}
\qed

\begin{theorem} \label{thm:compr}
Let $f_1,f_2,\ldots,f_t\in L^2(G)$ be complementary functions 
with $\paf$ constants $\alpha_0$ and $\alpha$. 
Then the functions  $f^M_1,f^M_2,\ldots,f^M_t\in L^2(H)$ are 
complementary with $\paf$ constants
\begin{equation} \label{paf-konst-zg}
\alpha^M_0=\alpha_0 +(m-1)\alpha, \quad \alpha^M=m\alpha.
\end{equation}
\end{theorem}
\begin{proof}
By Lemma \ref{le:f^M} we have
$$
{\paf_{f_i}}^M(xM)=\sum_{z\in M} \paf_{f_i} (xz).
$$
By adding these equation and by using the equation 
(\ref{eq:compl-psd=}), we obtain that
\begin{eqnarray*}
\sum_{i=1}^t {\paf_{f_i}}^M(xM) &=& \sum_{z\in M} 
\left((\alpha_0-\alpha)\delta_e(xz)+\alpha\theta(xz)\right) \\
&=& (\alpha_0-\alpha)\delta_M(xM) +m\alpha.
\end{eqnarray*}
(The delta function $\delta_M\in L^2(H)$ takes value 1 at the 
point $M\in H$ and 0 at all other points of $H$.)
Hence, the theorem is proved.
\end{proof}

\begin{corollary} \label{cor:zgusPSD}
Let $f_1,f_2,\ldots,f_t$ be as in the theorem. Then the 
$\psd$-constants $\beta_0^M$ and $\beta^M$ of 
$f_1^M,f_2^M,\ldots,f_t^M$ are the same as the 
$\psd$-constants $\beta_0$ and $\beta$ of $f_1,f_2,\ldots,f_t$.
\end{corollary}
\begin{proof}
This follows from the formulas (\ref{paf-konst-zg}) by applying 
Theorem \ref{thm:paf=psd} to the group $H$.
\end{proof}

An alternative proof of this corollary can be given by 
using the equation (\ref{eq:compl-psd}), the equalities

\begin{equation} \label{eq:dual}
\delta_\theta\circ\hat{\sigma}=\delta_{\theta_H}, \quad \hat{\theta}\circ\hat{\sigma}=\hat{\theta}_H
\end{equation}
 and the following lemma.

\begin{lemma} \label{le:kompr}
For $f\in L^2(G)$ we have 
$\dft_f \circ \hat{\sigma}=\dft_{f^M}$.
\end{lemma} 

\begin{proof}
Let $Y$ be a set of coset representatives of $M$ in $G$. Then 
for $\phi\in\hat{H}$ we have
\begin{eqnarray*}
\dft_f \circ \hat{\sigma} (\phi)
&=& \dft_f(\phi\sigma) \\
&=& \langle f,\phi\sigma \rangle \\
&=& \sum_{x\in G} f(x)\bar{\phi}(xM) \\
&=& \sum_{y\in Y} \sum_{z\in M} f(yz)\bar{\phi}(yM) \\
&=& \sum_{y\in Y} f^M(yM)\bar{\phi}(yM) \\
&=& \langle f^M,\phi \rangle \\
&=& \dft_{f^M}(\phi).
\end{eqnarray*}
\end{proof}

This lemma generalizes \cite[Theorem 3.1]{DK:D-optimal:2015} 
which applies only to the case when $G$ is cyclic. 
Thus, our lemma can be useful when one wants to construct 
a difference family in $G$ by using the compression method
\cite{DK:DCC:2014}.

\section{Regular representation} \label{sec:regular}

Since the algebra $\cR$ comes equipped with a natural basis, 
$G$, we can identify the algebra of linear transformations on 
$\cR$ with $M_v(\bC)$, the algebra of complex matrices of 
order $v$. For this one has to choose an ordering of $G$, 
however we will supress the ordering and will label the 
rows and columns of these matrices by the elements of $G$.
We have defined an involution on $\cR$, and there is also 
one on this matrix algebra, namely the conjugate transpose 
map. The regular representation of $\cR$ is the homomorphism of 
algebras with involution $\Mat:\cR\to M_v(\bC)$ which assigns 
to $X\in \cR$ the matrix of the linear transformation
$\cR\to \cR$ sending $Y\to XY$, $Y\in \cR$. In particular, 
we have $\Mat(X^*)=\Mat(X)^*$, $X\in \cR$.

The matrices $A=[a_{x,y}]\in\Mat(\cR)$ are {\em $G$-invariant}, 
i.e., they satisfy the condition
\begin{equation} \label{eq:gr-inv}
a_{xz,yz}=a_{x,y},\quad x,y,z\in A.
\end{equation}
For any $x\in G$ we have $Gx=G=\sum_{x\in G} x$. This 
implies that $\Mat(G)=J_v$, the all-one matrix of order $v$.

Let us use the hypotheses and notation of Theorem 
\ref{thm:compl-fun}. By setting $A_i=\Mat(a_{f_i})$ for
$i=1,2,\ldots,t$ and by applying $\Mat$ to the equation 
(\ref{eq:glavna}) we obtain the following matrix identity
\begin{equation} \label{eq:Mat-ident}
\sum_{i=1}^t A_iA_i^T = 4nI_v+(tv-4n)J_v.
\end{equation}
Note that $A_i^*=A_i^T$ because the $A_i$ are 
$\{\pm1\}$-matrices.

We shall need later the permutation matrix $R=[r_{x,y}]$ 
of order $v$ whose entries are defined by the formula 
$r_{x,y}=\delta_e(xy)$, $x,y\in G$. This matrix is 
involutory, $R^2=I_v$, but not $G$-invariant in general. 
However, for any $G$-invariant matrix $A=[a_{x,y}]$ we have 
\begin{equation} \label{eq:RAR}
RAR=A^T. 
\end{equation}
Indeed, the $(x,y)$-entry of $RAR$ is 
$\sum_{u,v\in G}\delta_e(xu)a_{u,v}\delta_e(vy)=
a_{x^{-1},y^{-1}}=a_{y,x}$.

\section{Some special classes of difference families} \label{sec:Classes}

For the construction of various type of Hadamard matrices and 
related combinatorial designs several special classes of 
difference families are widely used mostly over finite cyclic 
groups in which case they can be viewed as 
$\{\pm 1\}$-sequences. 
Many of them make sense over noncyclic finite abelian groups.
We list several such families and provide explicit examples.

\subsection{DO-matrices and DO-difference families $(\alpha=2)$}
\label{sub:DO}

Consider the $\{\pm1\}$-matrices $M$ of order $m$. Those among them which have maximum determinant are known as {\em D-optimal 
matrices} or {\em DO-matrices}. If $m=1,2$ or $m$ is a multiple 
of 4 then D-optimal matrices are just the Hadamard matrices. 
They have determinant $m^{m/2}$. When $m=2v$ with $v>1$ odd, 
then there are no Hadamard matrices of order $m$. In that case 
it is well known that 
$$
\det M\le 2^v (2v-1)(v-1)^{v-1}
$$
and that this inequality is strict if $2v-1$ is not a sum of 
two squares. We are interested here only in the case when 
$2v-1$ is a sum of two squares. Henceforth we assume in this 
section that $v$ satisfies this condition. 

For the known results on DO-matrices of order $m=2v$ we refer to 
\cite[V.3]{KO:2007} and \cite{DK:D-optimal:2015}. One can 
construct DO-matrices of order $2v$ from difference families 
$(X_1,X_2)$ in the group $G$ having the parameter sets 
$(v;r,s;\lambda)$, where 
$v/2\ge r\ge s\ge 1$ and $v=2n+1$, $n=r+s-\lambda$. 
We refer th these difference families as {\em DO-difference families}. Recall from Theorem \ref{thm:compl-fun} the 
associated functions $f_i=f_{X_i}$, $i=1,2$, and from section \ref{sec:regular} the matrices $A_i=\Mat(a_{f_i})$. Since in this case $t=2$ and $v=2n+1$, we have $\alpha=2$ and, 
by using the equation (\ref{eq:Mat-ident}), we obtain that $A_1A_1^T+A_2A_2^T=2(v+1)I_v-2J_v$. It follows that

\begin{equation} \label{eq:bicikl-1}
\left[ \begin{array}{cc} A_1 & A_2 \\ -A_2^T & A_1^T
\end{array} \right]
\end{equation}
is a DO-matrix. 

If the base block $X_1$ is symmetric, then the matrix $A_1$ is 
symmetric. Hence, if we multiply by $-1$ the second block-row 
in the above matrix, then that matrix becomes a symmetric 
DO-matrix
\begin{equation} \label{eq:bicikl-2}
\left[ \begin{array}{cc} A_1 & A_2 \\ A_2^T & -A_1^T
\end{array} \right].
\end{equation}

Let us give a simple example. In $G=\bZ_3\times\bZ_3$ we have 
the difference family $X_1=\{(0,0),(1,1),(2,1)\}$, 
$X_2=\{(0,1),(0,2)\}$ with the parameter set 
$(9;3,2;1)$, $n=4$, $v=9=2n+1$. Moreover, the block $X_2$ 
is symmetric. This gives a symmetric DO-matrix of order 18. 
We order the elements of $\bZ_3\times\bZ_3$
lexicographically. 
Since we used in previous sections the multiplicative notation for the group $G$, let us write $x$ for $(0,1)$ 
and $y$ for $(1,0)$. Then $x^3=y^3=e$ and $xy=yx$. 
The values of the function $f_1$ at the basis elements 
$e,x,x^2,y,xy,x^2y,y^2,xy^2,x^2y^2$ are 
$-1,1,1, 1,-1,1, -1,1,1$, respectively. The values 
of $f_2$ are $1,-1,-1, 1,1,1, 1,1,1$.
The corresponding elements $a_{f_i}\in \cR$ are 
\begin{eqnarray*}
a_{f_1} &=& -e+x+x^2+y-xy+x^2y+y^2-xy^2+x^2y^2, \\
a_{f_2} &=& e-x-x^2+y+xy+x^2y+y^2+xy^2+x^2y^2.
\end{eqnarray*}

Since $a_{f_1}e=a_{f_1}$, the first column of the matrix 
$A_1=\Mat(a_{f_1})$ is the transpose of the row 
$[-1,1,1, 1,-1,1, -1,1,1]$. Similarly, the first column of 
$A_2$ is the transpose of the row $[1,-1,-1, 1,1,1, 1,1,1]$. Further, by using some linear algebra, one finds that
$$
A_1=\left[ \begin{array}{ccc} 
P_1&P_2&P_3\\P_3&P_1&P_2\\P_2&P_3&P_1
\end{array} \right], \quad
A_2=\left[ \begin{array}{ccc} 
Q_1&Q_2&Q_3\\Q_3&Q_1&Q_2\\Q_2&Q_3&Q_1
\end{array} \right],
$$
where $P_1=J_3-2I_3$, $P_2=P_3=J_3-2C_3^T$,
$Q_1=2I_3-J_3$, $Q_2=Q_3=J_3$, and $C_3$ is the circulant matrix 
with first row $[0,1,0]$. 

We refer to matrices like $A_i$ as {\em multicirculants}. More 
generally, circulant matrices are multicirculants and, 
recursively, block-circulant matrices whose blocks are 
multicirculants are also multicirculants.

Since our block $X_2$ (and the matrix block $A_2$) is 
symmetric, we first switch $A_1$ and $A_2$ and then plug them into the array (\ref{eq:bicikl-2}). In this way we obtain 
a symmetric DO-matrix of order 18. One can easily verify that 
$A_1A_1^T+A_2A_2^T=16I_9+2J_9$.

\subsection{Periodic Golay pairs $(\alpha=0)$}
\label{sub:PerGol}

If $(X_1,X_2)$ is a difference family in $G$ with parameters 
$(v;k_1,k_2;\lambda)$ and $v=2n$ then, by Theorem 
\ref{thm:compl-fun}, the functions $f_i=f_{X_i}$, $i=1,2$, 
are complementary with \paf-constants $\alpha_0=2v$ and 
$\alpha=0$. When $G$ is cyclic, the functions $f_1$ and $f_2$ 
may be viewed as $\{\pm 1\}$-sequences 
$f_i(0),f_i(1),\ldots,f_i(v-1)$ and are known as {\em periodic 
Golay pairs}. For more information about such sequences and 
in particular the existence question for specified length $v$ 
see \cite{DK:PerGolay72:2015}. 

There is a necessary arithmetic condition for the existence 
of periodic Golay pairs of length $v$, a special case of a 
theorem of Arasu and Xiang \cite{Arasu:Xiang:DCC:1992}. 
This condition is not satisfied when $v=18$. Consequently there 
are no periodic Golay pairs of length 18. Equivalently, there are no cyclic difference familes with parameter set 
$(18;9,6;6)$. However, the theorem of Arasu and Xiang is applicable only to cyclic groups $G$ and may fail for other 
groups. 
Indeed, we have found that in the non-cyclic group 
$G=\bZ_3\times\bZ_6$ there exist difference families with 
parameter set $(18;9,6;6)$. Let us give an example
\begin{eqnarray*}
X_1 &=& \{ (0,0),(0,1),(0,4),(1,0),(1,2),(1,5),(2,2),(2,3),
(2,4) \}, \\
X_2 &=& \{ (0,0),(0,1),(0,3),(0,5),(1,0),(2,0) \}.
\end{eqnarray*}

Let $A_1$ and $A_2$ be the matrices defined as in section 
\ref{sec:regular}. With suitable indexing, these matrices are 
multicirculants and since $t=2$ and $v=2n$ the equation 
(\ref{eq:Mat-ident}) shows that $A_1A_1^T+A_2A_2^T=36I_{18}$. 
Hence, the matrix
\begin{equation*}
\left[ \begin{array}{cc} A_2 & A_1 \\ A_1^T & -A_2^T 
\end{array} \right]
\end{equation*}
is a symmetric Hadamard matrix of order 36 made up from two 
multicirculants.

\subsection{Legendre pairs $(\alpha=-2)$} \label{sub:LP}

If $q\equiv 3 \pmod{4}$ is a prime power then the nonzero 
squares in a finite field $F_q$ of order $q$ form a difference 
set in the additive group of $F_q$. The parameters of this 
difference set are $(q;(q-1)/2;(q-3)/4)$. If we use two 
copies of this difference set, we obtain a difference family 
with parameters $(q;(q-1)/2,(q-1)/2;(q-3)/2)$. 

By generalizing, we say that a difference family $(X_1,X_2)$ 
in $G$ (a finite abelian group of order $v$) having the 
parameter set $(v;(v-1)/2,(v-1)/2;(v-3)/2)$ is a 
{\em Legendre pair}. In the case when $G$ is cyclic, they were 
introduced first in the paper \cite{FGS:2001} where they were 
called ``generalized Legendre pairs". It was shown in the same paper \cite[Theorem 2]{FGS:2001} that such pairs give Hadamard matrices of order $2v+2$. Moreover if one of the blocks, say 
$X_1$, is symmetric or skew then the resulting Hadamard matrix can be made symmetric or skew-Hadamard, respectively, see the 
arrays $H_s$ and $H_k$ below. All these facts remain valid 
over arbitrary finite abelian groups.

Let $(X_1,X_2)$ be a Legendre pair with the above parameter set. 
The parameter $n$ for this parameter set is 
$n=(v-1)-(v-3)/2=(v+1)/2$. By Theorem \ref{thm:compl-fun} the 
associated functions $f_i$ of $X_i$, $i=1,2$, are complementary 
with \paf-constants $\alpha_0=2v$ and $\alpha=2v-4n=-2$.

As nontrivial examples we give two Legendre pairs in the 
non-cyclic group $\bZ_5\times\bZ_5$, which we identify with 
the additive group of the finite field $\bZ_5[x]/(x^2+2)$ 
of order 25. The block $X_1$ is symmetric in the first 
pair and skew in the second pair.

\begin{eqnarray*}
X_1 &=& \{\pm x,\pm 2x,\pm(1+2x),\pm(1+4x),\pm(2+3x),\pm(2+4x)\} \\
X_2 &=& \{1,3,1+2x,2+2x,4+x,4+4x,\pm(1+3x),\pm(3+x),\pm(3+2x)\};\\
\\
X_1 &=& \{1,2,x,2x,1+2x,1+3x,1+4x,2+3x,3+3x,2+4x,3+4x,4+4x\} \\
X_2 &=& \{1,3,4,2x,1+2x,1+4x,2+x,2+3x,3+2x,4+2x,4+3x,4+4x\}.
\end{eqnarray*}

The multicirculants $A_i$, $i=1,2$, associated with the base blocks $X_i$ of the first resp. second family should be plugged into the array $H_s$ resp. $H_k$ below to obtain a symmetric 
resp. skew Hadamard matrix of order 52.

\begin{center}
\begin{equation*}
H_{s} = \left[
\begin{array}{c c |c |c}
-  & -  & +~\cdots ~+ & +~ \cdots ~+ \\
-  & +  & +~\cdots ~+ & -~ \cdots ~- \\
\hline
+  & +  &              &     \\
\vdots  & \vdots & A_1 & A_2 \\
+  & +  &              &     \\
\hline
+  & -  &                &  \\
\vdots  & \vdots & A_2^T & -A_1^T \\
+  & -  &                &  \\
\end{array}
\right], \quad
H_{k} = \left[
\begin{array}{c c |c |c}
+  &  -  & +~\cdots ~+ & +~ \cdots ~+ \\
+  &  +  & +~\cdots ~+ & -~ \cdots ~- \\
\hline
-  & -  &                &        \\
\vdots  & \vdots & A_1 & A_2 \\
-  & -  &                &        \\
\hline
-  & +  &                &     \\
\vdots  & \vdots & -A_2^T   & A_1^T \\
-  & +  &                &     \\
\end{array}
\right] 
\end{equation*}
\end{center}

\section{Goethals-Seidel quadruples $(\alpha=0)$} \label{sec:GS}

In our definition of difference families in a finite abelian 
group $G$ of order $v$, we have the rather unnatural 
restriction that no base block $X_i$ can be $\emptyset$ or $G$.
The usual justification for that restriction is that such 
blocks are trivial. The trivial blocks can be discarded and 
the number, $t$, of base blocks lowered. However in some applications one does not have the freedom of changing the 
parameter $t$. For that reason, it is necessary to allow 
the possibility of trivial blocks in such applications. In this 
section we shall examine one such application. To avoid confusion, we warn the reader that in the remainder of this section we deviate from Definition \ref{def:diff-fam} by 
permitting the base blocks to be trivial. Since we deal with the 
``trivial'' cases, we also permit $G$ to be trivial, i.e., 
we may have $v=1$.

We say that a difference family $(X_0,X_1,X_2,X_3)$ is a 
{\em Goethals-Seidel family} (or quadruple) if its 
parameters $(v;k_0,k_1,k_2,k_3;\lambda)$ and 
$n=\sum k_i -\lambda$ satisfy the additional condition $n=v$.
This additional condition is equivalent to $\alpha=0$, 
see Theorem \ref{thm:compl-fun}. It is also equivalent to

\begin{equation} \label{eq:trivial}
\sum_{i=0}^3 k_i=\lambda+v.
\end{equation}

For convenience, we arrange here the $k_i$ so that
\begin{equation} \label{eq:ineq}
0\le k_0\le k_1\le k_2\le k_3\le v/2.
\end{equation}

By plugging the associated matrices $A_i$ into the well-known Goethals--Seidel array (GS-array)
\begin{equation} \label{GS-matrix}
H=\left[ \begin{array}{cccc}
A_0 & A_1R & A_2R & A_3R \\
-A_1R & A_0 & -A_3^T R & A_2^T R \\
-A_2R & A_3^T R & A_0 & -A_1^T R \\
-A_3R & -A_2^T R & A_1^T R & A_0
\end{array} \right],
\end{equation}
we obtain a Hadamard matrix $H$. This can be easily verified by using the equation (\ref{eq:RAR}).

If some $k_i=0$ then $X_i=\emptyset$ and the corresponding 
matrix block $A_i=J_v$, the all-one matrix of order $v$. 
We are here interested only in the cases when at least one of 
the matrix blocks $A_i$ is equal to $J_v$. As we assume 
that (\ref{eq:ineq}) holds, this means that $k_0=0$. 
Then the equation (\ref{eq:lambda}) takes the form 
\begin{equation} \label{eq:k0=0}
\sum_{i=1}^3 k_i(k_i-1)=\lambda(v-1).
\end{equation}

If $v=1$ then this equation says nothing about $\lambda$
and we shall use equation (\ref{eq:trivial}) to compute 
$\lambda$. The equation (\ref{eq:k0=0}) can be written as
\begin{equation} \label{eq:4-v}
\sum_{i=1}^3 \left( \frac{v}{2}-k_i \right)^2=
\frac{v}{4}(4-v).
\end{equation}
It implies that $v\le4$. Thus there are four cases to 
consider according to the value of $v=1,2,3,4$. In the 
first three cases the group $G$ is necessarily cyclic 
and we shall assume that $G=\bZ_v$.

Case $v=1$. Then (\ref{eq:ineq}) implies that 
$k_1=k_2=k_3=0$. Thus all $X_i=\emptyset$, and all four blocks $A_i=J_1=I_1$, the identity matrix of order 1. 
We can plug these blocks into the 
GS-array to obtain a Hadamard matrix of order 4.
(Note that in this case the equation (\ref{eq:trivial}) 
implies that $\lambda=-1$.)

Case $v=2$. Then the inequalities (\ref{eq:ineq}) and
the equation (\ref{eq:4-v}) imply that $k_1=0$ and 
$k_2=k_3=1$. Now $X_0=X_1=\emptyset$ and we can set 
$X_2=X_3=\{0\}$. Thus $A_0=A_1=J_2$ and 
$A_2=A_3=J_2-2I_2$. We can plug these blocks into the 
GS-array to get a Hadamard matrix of order 8.
(In this case $\lambda=0$.)

Case $v=3$. The inequalities (\ref{eq:ineq}) and
the equation (\ref{eq:4-v}) imply that $k_1=k_2=k_3=1$. 
Now $X_0=\emptyset$ and we can set $X_1=X_2=X_3=\{0\}$. 
The four matrix blocks are $A_0=J_3$ and 
$A_1=A_2=A_3=J_3-2I_3$. This gives a Hadamard matrix of 
order 12. (In this case again $\lambda=0$.)

Case $v=4$. The inequalities (\ref{eq:ineq}) and
the equation (\ref{eq:4-v}) imply that $k_1=k_2=k_3=2$. 
It follows that $\lambda=2$.

If $G=\bZ_4$, a cyclic group, there is a GS-difference 
family with $X_0=\emptyset$ and $X_i=\{0,i\}$ for 
$i=1,2,3$.

There is also a possibility that $G$ is a Klein four-group, 
which we identify with the additive group $\{0,1,x,1+x\}$,
$x^2=1+x$, of the finite field $F_4$ of order 4. The 
required GS-family exists, e.g.
$$
X_0=\emptyset,~X_1=\{0,1\},~X_2=\{0,x\},~X_3=\{0,1+x\}.
$$

For this difference family we give all details for the construction of the corresponding Hadamard matrix. 
If we label the rows and columns of the matrix blocks $A_i$ 
with group elements $0,1,x,1+x$ then $A_0=J_4$ and
$$
A_1=\left[ \begin{array}{cccc}
-&-&+&+\\ -&-&+&+\\ +&+&-&-\\ +&+&-&- \end{array} \right], \quad
A_2=\left[ \begin{array}{cccc}
-&+&-&+\\ +&-&+&-\\ -&+&-&+\\ +&-&+&- \end{array} \right], \quad
A_3=\left[ \begin{array}{cccc}
-&+&+&-\\ +&-&-&+\\ +&-&-&+\\ -&+&+&- \end{array} \right],
$$
where we write $\pm$ insted of $\pm1$.
Since $u+u=0$ for all $u\in G$, we have $\delta_{0,y+z}=1$ if 
and only if $y=z$. Thus $R=I_4$ in this case. For the same reason, each $X_i$ (and each $A_i$) is symmetric. By plugging 
the $A_i$ into the GS-array, we obtain the symmetric Hadamard 
matrix
$$
H=\left[ \begin{array}{cccc}
A_0 & A_1 & A_2 & A_3 \\
A_1 & A_0 & A_3 & A_2 \\
A_2 & A_3 & A_0 & A_1 \\
A_3 & A_2 & A_1 & A_0 \end{array} \right].
$$
It is easy to check that $A_0^2=4A_0$ and $A_i^2=-4A_i$ for 
$i\ne 0$. Further we have $A_iA_j=0$ whenever $i\ne j$ and 
$A_0-A_1-A_2-A_3=4I_4$. It is now easy to verify that $H$ is 
indeed a Hadamard matrix.

This $H$ is an example of a Hadamard matrix of Bush-type 
(see \cite{CK-Had:2007}) because 
its order is a square $4v=m^2$, when partitioned into blocks 
of size $m$, the diagonal blocks are all equal to $J_m$ and each 
off-diagonal block has all row and column sums 0.

In the case $G=\bZ_4$ we also obtain a Bush-type Hadamard matrix 
of order 16. (The details are left to the reader.)

\end{document}